\documentclass{mcom-l}

\usepackage{multicol}
\usepackage{float}
\usepackage[margin=0.5cm]{caption}
\usepackage{graphicx}
\usepackage{geometry}
\usepackage{amsmath}
\usepackage{amsfonts,amssymb,amsthm}
\usepackage{newpxtext} 
\usepackage{amsthm}

\usepackage{picinpar,moresize,xfrac,graphpap,dcolumn,wrapfig,graphicx}
\DeclareGraphicsExtensions{.png,.pdf,.jpg}
\usepackage{subcaption}
\usepackage{stackengine}

\usepackage{tikz,pgfplots,tikz-cd}
\usepgfplotslibrary{colormaps}
\usetikzlibrary{pgfplots.colormaps}
\usepackage{rotating}
\pgfplotsset{compat=newest}
\pgfplotsset{plot coordinates/math parser=false}
\newlength\figureheight
\newlength\figurewidth

\usepackage{soul,array,calc,url,ragged2e,graphpap}
\usepackage{booktabs} 
\usepackage{tabularx}
\usepackage{colortbl}
\usepackage{multirow}
\usepackage{hyphenat}
\usepackage{transparent}
\usepackage{enumerate}

\usepackage{hyperref}
\usepackage{hypcap}
\usepackage{mathtools}
\mathtoolsset{centercolon}
\usepackage{nicefrac}
\usepackage{units}
\usepackage[normalem]{ulem}
\usepackage{cancel}
\usepackage{pbox}

\makeatletter
\DeclareFontFamily{OMX}{MnSymbolE}{}
\DeclareSymbolFont{MnLargeSymbols}{OMX}{MnSymbolE}{m}{n}
\SetSymbolFont{MnLargeSymbols}{bold}{OMX}{MnSymbolE}{b}{n}
\DeclareFontShape{OMX}{MnSymbolE}{m}{n}{
    <-6>  MnSymbolE5
   <6-7>  MnSymbolE6
   <7-8>  MnSymbolE7
   <8-9>  MnSymbolE8
   <9-10> MnSymbolE9
  <10-12> MnSymbolE10
  <12->   MnSymbolE12
}{}
\DeclareFontShape{OMX}{MnSymbolE}{b}{n}{
    <-6>  MnSymbolE-Bold5
   <6-7>  MnSymbolE-Bold6
   <7-8>  MnSymbolE-Bold7
   <8-9>  MnSymbolE-Bold8
   <9-10> MnSymbolE-Bold9
  <10-12> MnSymbolE-Bold10
  <12->   MnSymbolE-Bold12
}{}
\let\llangle\@undefined
\let\rrangle\@undefined
\DeclareMathDelimiter{\llangle}{\mathopen}%
                     {MnLargeSymbols}{'164}{MnLargeSymbols}{'164}
\DeclareMathDelimiter{\rrangle}{\mathclose}%
                     {MnLargeSymbols}{'171}{MnLargeSymbols}{'171}
\makeatother

\usepackage{fancyvrb,listings}
\usepackage{algorithm}
\usepackage{algorithmicx}
\usepackage{algpseudocode}

\let\ForEach\ForAll

\algrenewcommand\alglinenumber[1]{\footnotesize #1:}
\makeatletter  
 \renewcommand{\ALG@name}{\small Algorithm} 
\makeatother 

\theoremstyle{definition}
\newtheorem{theorem}{Theorem}
\newtheorem{lemma}{Lemma}

\newtheorem{corollary}{Corollary}
\newtheorem{definition}{Definition}
\newtheorem{example}{Example}
\newtheorem{remark}{Remark}

\newcommand{\figref}[1]{Figure~\ref{#1}}

\newcommand{\appref}[1]{Appendix~\ref{#1}}

\newcommand{\thmref}[1]{Theorem~\ref{#1}}

\def\RR{\mathbb{R}}

\def\cK{\mathcal{K}}

\usepackage{bbm}
\DeclareSymbolFont{bbold}{U}{bbold}{m}{n}
\DeclareSymbolFontAlphabet{\mathbbold}{bbold}

\newcommand\Inv{{\textrm{Inv}}}
\newcommand\Mobius{\text{M\"obius}}

\AtBeginDocument{\hypersetup{pdfborder={0 0 0}}}

\renewenvironment{quote}
  {\list{}{\rightmargin=0.5cm \leftmargin=0.5cm}%
   \item\relax}
  {\endlist}

\numberwithin{equation}{section}

\begin{document}

\title{Wave Simulations in Infinite Spacetime}


\author{Chad McKell}
\curraddr{}
\thanks{}

\author{Mohammad Sina Nabizadeh}
\curraddr{}
\thanks{}

\author{Stephanie Wang}
\curraddr{}
\thanks{}

\author{Albert Chern}
\curraddr{}
\thanks{}


\date{}

\dedicatory{}

\begin{abstract}
Solving the wave equation on an infinite domain has been an ongoing challenge in scientific computing. Conventional approaches to this problem only generate numerical solutions on a small subset of the infinite domain. In this paper, we present a method for solving the wave equation on the entire infinite domain using only finite computation time and memory. Our method is based on the conformal invariance of the scalar wave equation under the Kelvin transformation in Minkowski spacetime. As a result of the conformal invariance, any wave problem with compact initial data contained in a causality cone is equivalent to a wave problem on a bounded set in Minkowski spacetime. We use this fact to perform wave simulations in infinite spacetime using a finite discretization of the bounded spacetime with no additional loss of accuracy introduced by the Kelvin transformation.
\end{abstract}

\maketitle

\begin{multicols}{2}

Numerically solving the wave equation is a fundamental component in many areas of computational science, including acoustics, 
optics, and seismology~\cite{mehra2014acoustic, huang1991scalar, boore1972finite}.
Although the wave equation and its numerical schemes are long studied classical subjects, there remain computational challenges.  
An outstanding example is the \emph{infinite domain problem}---to simulate wave propagation on an infinite domain---frequently encountered by applications involving an open spacetime.
On an unbounded domain, it is intractable, for example, to use any spacetime discretization with a consistent resolution throughout the entire spacetime.
Thus, in practice, these domains have to be truncated to bounded ones when spacetime is discretized.  This compromise has driven a major line of research into developing artificial non-reflecting boundary conditions (NRBCs) accompanying domain truncation for the last 50 years~\cite{orlanski1976simple, engquist1977absorbing, fix1978variational, berenger1994perfectly, chern2019reflectionless}. 

In this paper, we call attention to a non-trivial symmetry within the wave equation:
The scalar wave equation \(\frac{\partial^2 u}{\partial t^2} = \Delta u \) is \emph{conformally invariant} on the Minkowski spacetime when the scalar field \(u\) is treated as a suitable \emph{fractional density}.
It is analogous to the \emph{Kelvin transformation} invariance of the Laplace equation in Euclidean space~\cite{nabizadeh2021kelvin}, but now with a Minkowski metric signature. 
Taking advantage of this conformal geometric property of the wave equation gives elegant solutions to some challenges that do not persist under conformal transformations.

For example, the boundedness of a domain in spacetime is not conformally invariant---a conformal spacetime transformation can turn unbounded subsets to bounded ones.
Therefore, for the infinite domain problem, we can first map the domain to a bounded one conformally, discretize this bounded domain, solve the wave equation by standard numerical schemes, and finally map the solution back to the original infinite domain (\figref{fig:hero_image}). 
No domain truncation is ever needed throughout the process.
We show that:
\begin{quote}
    \emph{Every wave propagation problem on any spacetime domain contained in a causality cone, despite possibly being infinite both spatially and temporally, is equivalent to a wave propagation problem on a compact domain in spacetime (\figref{fig:IllustrationMainStatement}). }
\end{quote}

This proposition implies that simulating wave propagation in the entire infinite spacetime domain only requires simulating wave computations on a bounded domain.  In particular, the computational cost and accuracy are only related to a bounded domain problem. 

\begin{figure}[H]
    \centering
    \includegraphics[width=0.95\linewidth]{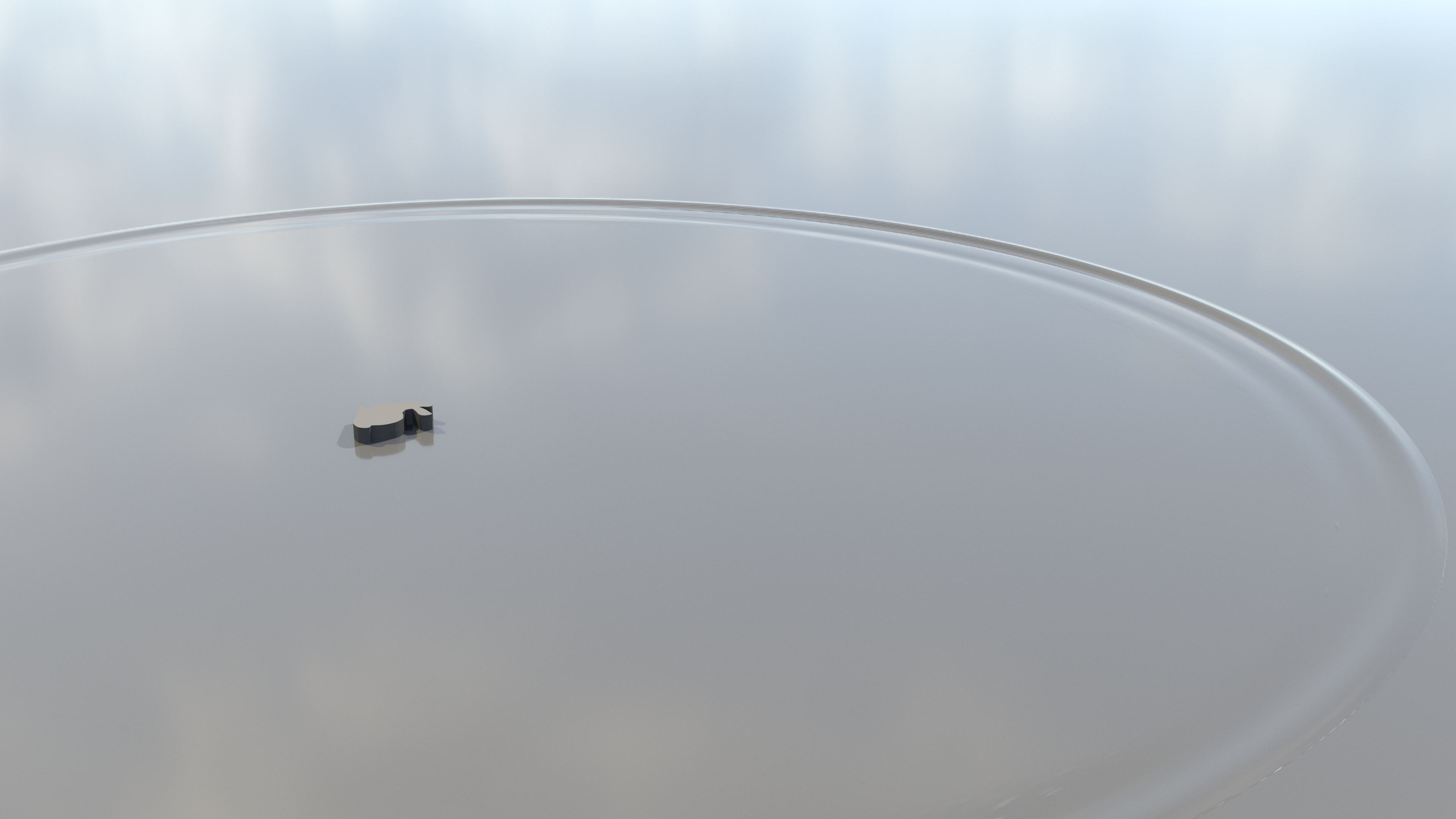}
    \captionof{figure}{Our solution to the infinite domain problem at an instance of time. Our method simulates wave propagation over the infinite expanse of a Minkowski spacetime.}
    \label{fig:hero_image}
\end{figure}

\section{Related Work}

The first documented instance of conformal invariance for wave-like equations on Minkowski spacetime are the Maxwell equations, as discovered by Cunningham and Batemann \cite{Cunningham:1910:PRE,Bateman:1910:TEE}.  
The corresponding spacetime transformations are the \emph{spherical wave transformations} studied in the Lie sphere geometry \cite{cecil2008lie}.  These results have been expanded to other wave equations \cite{Dirac:1936:WEC,Mclennan:1956:CIC,Lomont:1961:CIM}, including the scalar wave equation.  However, the conformal invariance of the wave equation has not been applied in the development of numerical methods for waves computations.

Over the past half-century, several numerical methods have been developed to handle infinite domains in wave simulations. To highlight the distinct advantages of our approach over earlier techniques, we briefly discuss two conventional approaches to the infinite domain problem. For a more comprehensive review of these conventional approaches, see~\cite{tsynkov1998numerical, seibel2022boundary}.

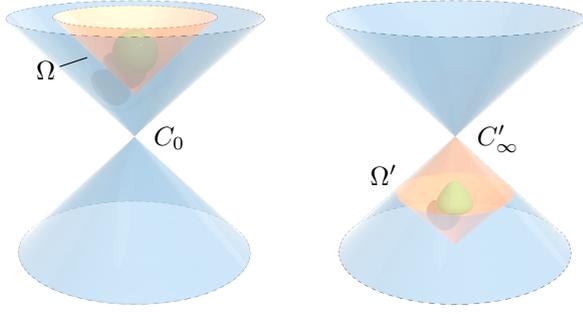
\begin{figure}[H]
    \centering
    \begin{picture}(240,120)
    \put(0,0){\includegraphics[width=0.48\columnwidth, trim = 300px 0px 300px 0px, clip]{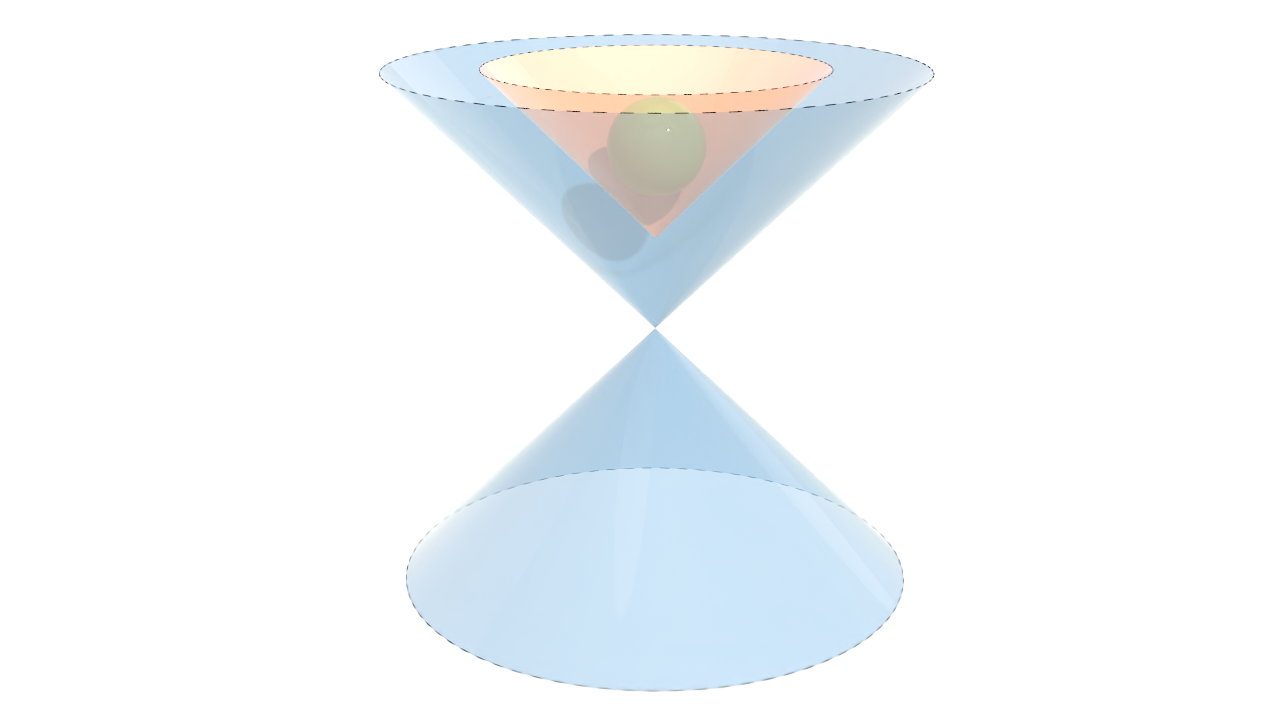}}
    \put(120,0){\includegraphics[width=0.48\columnwidth, trim = 300px 0px 300px 0px, clip]{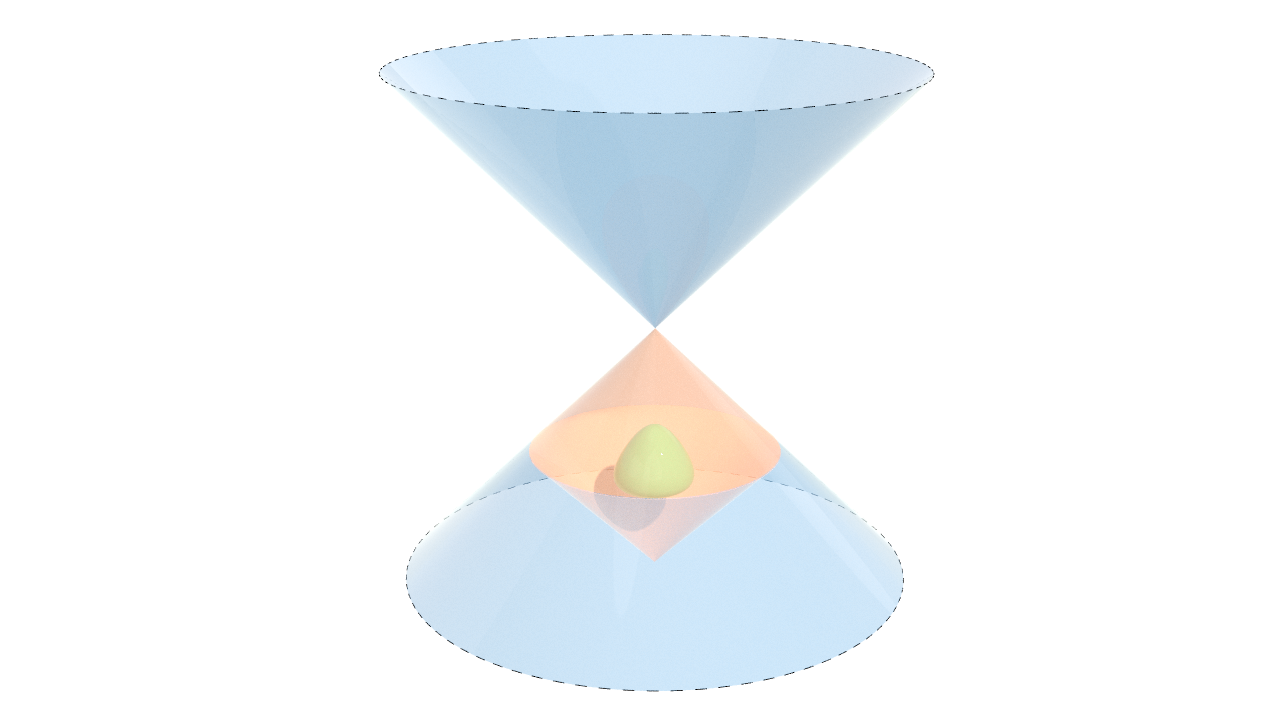}}
    \put(25,90){\(\Omega\)}
    \put(44,98){\line(-3,-1){10}}
    \put(69,65){\(C_{\rm 0}\)}
    \put(190,65){\(C'_{\rm \infty}\)}
    \put(150,50){\(\Omega'\)}
    \end{picture}
    \captionof{figure}{A wave problem on an unbounded Minkowski spacetime domain $\Omega$ contained in a causality cone $C_0$ (left) is equivalent to a wave problem on a bounded Minkowski spacetime domain $\Omega'$ contained in a light cone $C'_\infty$ which is mapped from infinity~(right).}
    \label{fig:IllustrationMainStatement}
    
\end{figure}

The standard numerical approach for solving the infinite domain problem is \emph{domain truncation}~\cite{tsynkov1998numerical}. In domain truncation, one extracts a finite region from the infinite spacetime domain and then numerically solves the scalar wave equation at discrete points on the extracted region using a numerical integration method such as the finite-difference time-domain~(FDTD) method~\cite{yee1966numerical}. Often, artificial NRBCs are imposed along the outer boundary of the extracted region in order to efficiently suppress spurious reflections at the boundary. Notable examples of NRBCs include Sommerfeld radiation boundary conditions~\cite{orlanski1976simple}, absorbing boundary conditions~\cite{engquist1977absorbing}, transparent boundary conditions~\cite{fix1978variational}, and perfectly matched layers~\cite{berenger1994perfectly, chern2019reflectionless}. In addition to the expense of imposing NRBCs, discretizing large multi-dimensional spacetime regions incurs a high computational cost. When the discretization is performed using the FDTD method, for example, the units of storage and floating-point operations both scale as $\mathcal{O}(MN)$, where $M$ is the number of spatial grid points and $N$ is the number of time steps. As $M$ and $N$ grow larger, computational resources will eventually be exhausted. 

One may alternatively solve the infinite domain problem using \emph{time-domain boundary element methods} (TD-BEMs), which only require the discretization of the obstacle boundary instead of the spacetime domain~\cite{seibel2022boundary}. In this approach, one numerically computes the boundary data that satisfy the representation formula for the scalar wave equation. Given a finite number of obstacle boundary elements~$M$ that represent the initial data on a three-dimensional obstacle boundary, one typically requires $\mathcal{O}(M^2N)$ units of storage and $\mathcal{O}(M^2N^2 + M^3)$ operations to simulate wave propagation. With sophisticated compression and approximation algorithms, these costs can be reduced to $\mathcal{O}(M + N)$ and $\mathcal{O}(MN + M^2)$, respectively~\cite{seibel2022boundary}. However, regardless of how TD-BEMs are implemented, the time component of the wave problem must be truncated in order to produce a finite value of~$N$. Therefore, only a partial solution to the infinite domain problem is feasible.

With our approach, we can simulate wave propagation on an untruncated infinite spacetime domain using only finite spatial grid elements and finite time steps. Our method is based on the classical Kelvin transformation discovered by William Thomson (Lord Kelvin) in 1845~\cite{thomson1845extrait}. To simulate physical phenomena on an infinite domain, the Kelvin transformation conformally maps a Laplace problem defined on an infinite domain to an arbitrary bounded domain via a M\"obius inversion. In Section~\ref{sec:mobius_invariance}, we formulate the scalar wave equation as a Laplace problem in Minkowski spacetime and show that the problem is conformally invariant under a Kelvin transformation using the Minkowski metric. By extension, any wave problem on an infinite Minkowski spacetime contained in a causality cone is equivalent to a wave problem on a bounded Minkowski spacetime. In Section \ref{sec:InfDomainProb}, we use this fact to simulate wave propagation in infinite spacetime using a finite discretization of the bounded spacetime.

\section{M\"obius Invariance of the Wave Equation} \label{sec:mobius_invariance}
\subsection{Wave equation as Laplace problem in Minkowski spacetime}
Let $\mathbb R^{n,1} = \{(\vec x, t): \vec x \in \mathbb R^n, t \in \mathbb R\}$ denote the $(n+1)$-dimensional spacetime domain. The wave equation%
\footnote{By re-scaling time $t\gets \nicefrac tc$, we omit the mentioning of propagation speed $c>0$.}%
of is a second-order partial differential equation a scalar field $u$,
\begin{equation}\label{eq:wavePDE}
    \square_{(\vec x,t)} u = \bigg(\sum_{i=1}^{n}\frac{\partial^2 u}{\partial x_i^2}\bigg) -\frac{\partial^2 u}{\partial t^2} = 0.
\end{equation}
The operator $\square_{(\vec x,t)} = \big(\sum_{i=1}^n \frac{\partial^2}{\partial x_i^2}\big) - \frac{\partial^2}{\partial t^2}$ is called the \emph{d'Alembertian}. One may equip $\RR^{n,1}$ with a \emph{Minkowski metric tensor} $\eta$: 
\begin{equation} \label{eq:Minkowski_metric}
\eta((\vec u,a),(\vec v,b)) \coloneqq \langle \vec u,\vec v\rangle_{\RR^n} - ab,
\end{equation}
where $(\vec u, a), (\vec v, b)\in\RR^{n,1}$ are arbitrary spacetime vectors ($\vec u, \vec v \in \RR^n$, $a,b\in\RR$). 
Together, $(\RR^{n,1}, \eta)$ is referred to as the $(n+1)$-dimensional \emph{Minkowski spacetime}. 
Under this metric,
the Laplace operator is the d'Alembertian \(\square_{(\vec x,t)}\) and \eqref{eq:wavePDE} can be interpreted as a Laplace problem.

The analogy between the wave equation and the Laplace problem allows us to explore symmetries in wave equations similar to those in Laplace equations. In particular, the following theory is analogous to the invariance of the Laplace equation under Kelvin transformations in Euclidean space.

\subsection{M\"obius transformations}
A map \(\Phi\colon \Omega\to\Omega'\) between two spacetime open sets $\Omega, \Omega' \subset \mathbb R^{n,1}$ is a \emph{conformal map} if it preserves the metric tensor $\eta$ up to a (possibly non-constant) scalar factor $\varphi: \Omega \rightarrow \mathbb R_{>0}$. In other words, $\Phi$ is conformal if and only if for all $(\vec x, t)\in \Omega$ and $(\vec u, a), (\vec v, b) \in \mathbb R^{n,1}$, 
\begin{equation}\label{eq:pullbackmetric}
    \eta( d\Phi_{(\vec x, t)}(\vec u, a), d\Phi_{(\vec x, t)}(\vec v, b)) = \varphi(\vec x,t)^2 \eta((\vec u, a), (\vec v, b)).
\end{equation}
The left-hand side of \eqref{eq:pullbackmetric} is referred to as the \emph{pullback metric by $\Phi$}, and the scalar $\varphi$ the \emph{conformal factor}. One can verify that the Minkowski inversion, defined by 
\begin{align}\label{eq:Inv}
\Inv\colon
    (\vec x,t)\mapsto \left({\frac{\vec x}{ |\vec x|_{\RR^n}^2 - t^2}}, {\frac{t}{|\vec x|_{\RR^n}^2 - t^2}}\right),
\end{align}
is indeed a conformal map, with conformal factor
\begin{align} \label{eq:Inv_factor}
    \varphi_\Inv(\vec x,t) = \left|{\frac{1}{|\vec x|_{\RR^n}^2 - t^2}}\right|.
\end{align}
For a proof of the conformality of the Minkowski inversion map, see \appref{app:inversionConformalityProof}. 
The inversion as in \eqref{eq:Inv} is defined only for points away from the \emph{light cone} \(C_0=\{(\vec x,t)\in\RR^{n,1}\,\vert\,|\vec x|_{\RR^n}^2-t^2=0\}\). 
One can extend the definition of $\Inv$ for points on the light cone by considering the \emph{one-cone compactification} of the Minkowski spacetime (\appref{app:ExtendedMinkowskiSpacetime}). 

The set of M\"obius transformations on $\mathbb R^{n,1}$, $\Mobius(n,1)$, is the group generated by the composition of finitely many Lorentz transformations (linear isometries in \(\RR^{n,1}\)), translations (\((\vec x,t)\mapsto(\vec x+\vec a, t+b)\)), scalings (\((\vec x,t)\mapsto (s\vec x,st)\)), and Minkowski inversions. 
If $\Phi_1:\Omega \rightarrow \Omega'$ and $\Phi_2:\Omega' \rightarrow \Omega''$ are both conformal, then $\Phi_2\circ\Phi_1: \Omega\rightarrow\Omega''$ is also conformal, and its conformal factor is $\varphi^{\Phi_2\circ\Phi_1} = (\varphi^{\Phi_2}\circ\Phi_1) \varphi^{\Phi_1}$.
Since all of Lorentz transformations, translations, scaling, and Minkowski inversion are conformal, all mappings in $\Mobius(n,1)$ are conformal. 

As established in the Lie sphere geometry~(\cite{Klein:1926:VUH}, \cite[Chapter~2]{cecil2008lie}) and the classification of conformally flat spacetime~\cite{cahen1983domaines}, there are \(\nicefrac{(n+3)(n+2)}{2}\) degrees of freedom of conformal maps on the compactified Minkowski spacetime, all of which are from $\Mobius(n,1)$.

\subsection{M\"obius invariance of waves}
Here, we describe the invariance of the wave problem \eqref{eq:wavePDE} under M\"obius transformations on the domain.
Let $\Omega'\subset\RR^{n,1}$ be a spacetime domain.
Given $\Phi\in\Mobius(n,1)$ with conformal factor $\varphi\colon\Omega'\to\RR_{>0}$, define the \emph{Minkowski--Kelvin transform} 
$\mathcal{K}_{\Phi}\colon (\Omega\to\RR)\xrightarrow{\rm linear} (\Omega'\to\RR)$ for $u:\Omega\rightarrow \RR$:

\begin{equation} \label{eq:minkowskiKelvinTransform}
    \mathcal{K}_{\Phi}u\coloneqq \varphi^{\frac{n-1}{2}}u\circ \Phi,
\end{equation}
where $\Omega = \Phi(\Omega')\subset\RR^{n,1}$ is the image of \(\Omega'\) under \(\Phi\).

To clarify the domain of each function symbolically, the coordinates for \(\Omega'\) are denoted by \((\vec \xi,\tau)\) while the coordinates for \(\Omega\) are \((\vec x,t)\). 
\begin{equation}
\label{eq:map_diagram}
\begin{tikzcd} 
\Omega'\subset\RR^{n,1}_{(\vec \xi, \tau)} 
\arrow[d, "\mathcal{K} u", shift left=0.75em] \arrow[r, "\Phi"] \arrow[d, "\varphi", swap, shift right=0.75em] & 
\Omega\subset\RR^{n,1}_{(\vec x, t)} 
\arrow[d, "u"]\\
\RR_{>0} \, \RR & \RR
\end{tikzcd}
\end{equation}
As stated in the theorem below, the wave equation is \emph{Minkowski--Kelvin invariant} in the sense that \(\square_{(\vec\xi,\tau)}(\mathcal{K}u) = 0\) if and only if \(\square_{(\vec x,t)} u = 0\), 
where \(\square_{(\vec\xi,\tau)} = \big(\sum_{i=1}^n\frac{\partial^2}{\partial\xi_i^2}\big) - \frac{\partial^2}{\partial\tau^2}\) and \(\square_{(\vec x,t)} = \big(\sum_{i=1}^n\frac{\partial^2}{\partial x_i^2}\big) - \frac{\partial^2}{\partial t^2}\) are the standard d'Alembertian in the respective coordinates.

\begin{theorem}
\label{thm:MinkowskiKelvin}
Let $\Phi:\RR^{n,1}_{(\vec \xi, \tau)}\rightarrow \RR^{n,1}_{(\vec x, t)}$ be M\"obius with conformal factor $\varphi(\vec \xi, \tau)$ and let $u(\vec x, t)$ be a function. If $\square_{(\vec x,t)} u(\vec x, t) = 0$, then 

\begin{equation}\label{eq:mk_transform_invariance}
\square_{(\vec \xi,\tau)} (\mathcal K_\Phi u)(\vec\xi,\tau) = \square_{(\vec \xi,\tau)} (\varphi^{\frac{n-1}2} u\circ \Phi)(\vec\xi,\tau) = 0.
\end{equation}
\end{theorem}

\begin{proof}
\appref{app:conformWaveEqnMinkowski}.
\end{proof}

\thmref{thm:MinkowskiKelvin} can also be interpreted as that the wave equation is M\"obius invariant when the function \(u\) is transformed like a \emph{\(\frac{1}{2}\frac{n-1}{n+1}\)-density}, which obeys the scaling law given by \eqref{eq:minkowskiKelvinTransform}.

The significance of \thmref{thm:MinkowskiKelvin} is that it allows one to map two wave problems to each other. These two problems share the same equation to solve (i.e. \eqref{eq:wavePDE}, constituted with the same standard d'Alembertian \(\square\)), but the problems can be defined on two vastly different spacetime domains \(\Omega',\Omega\) related by a M\"obius transformation.  
In particular, a M\"obius transformation on the Minkowski space is capable of mapping an unbounded domain to a bounded one.
Suppose the domain \(\Omega\subset\RR^{n,1}_{(\vec x,t)}\) is contained in any \emph{causality cone}, that is, there exists an apex  \((\vec p_0,a_0)\in\RR^{n,1}\) such that 
\begin{align}
\label{eq:ContainedInCausalityCone}
    \Omega\subset \left\{(\vec x,t)\in\RR^{n,1}_{(\vec x,t)}\,\middle|\,
    |\vec x-\vec p_0|_{\RR^n}^2 < (t-a_0)^2, t>a_0
    \right\}.
\end{align}
Note that \(\Omega\) may still be spatially and temporally unbounded in \(\RR^{n,1}_{(\vec x,t)}\).
Then, by translating the apex to \((\vec 0,-\epsilon)\) for any \(\epsilon>0\) followed by the inversion \eqref{eq:Inv}, the transformed domain
\begin{align}
\label{eq:ContainedInCausalityCone_bounded}
    \Omega' = \Inv(\Omega-(\vec p_0,a_0+\epsilon))\subset\RR^{n,1}_{(\vec \xi,\tau)}
\end{align}
is a bounded set in \(\RR^{n,1}_{(\vec\xi,\tau)}\).

\begin{corollary} \label{cor1}
A wave problem \eqref{eq:wavePDE} on any domain \(\Omega\subset\RR^{n,1}\) contained in any causality cone \eqref{eq:ContainedInCausalityCone} is equivalent to a wave problem \eqref{eq:wavePDE} on some bounded domain \(\Omega'\subset\RR^{n,1}\) (\figref{fig:IllustrationMainStatement}).
\end{corollary}

The next section provides an explicit mapping between the two problems for applications on infinite domains.

\section{The Infinite Domain Problem} \label{sec:InfDomainProb}

In this section, we apply Corollary~\ref{cor1} to the initial boundary value problem (IBVP) for simulating wave propagation on infinite domains. 

\subsection{Unbounded form}
Our infinite domain IBVP takes the following general form.  Let \(\RR^n\setminus B\) be the spatial domain for some bounded set (possibly empty) \(B\subset\RR^{n}\) as obstacles.  
Given compactly supported functions \(f,h\colon \RR^n\setminus B\to\RR\) as initial data, and \(p\colon \partial B\to \RR\) as boundary data,
solve for \(u\colon \RR^n\times[t_0,\infty)\to \RR\) that satisfies
\begin{align}
\label{eq:IBVPOriginal}
    \begin{cases}
    \square_{(\vec x,t)}u(\vec x,t) = 0,&\vec x\in \RR^n\setminus B, t\geq t_0,\\
    (\hat Lu)(\vec x,t) = p(\vec x),& \vec x\in\partial B, t> t_0,\\
    u(\vec x,t_0) = f(\vec x),& \vec x\in \RR^n\setminus B,\\
    \frac{\partial u}{\partial t}(\vec x, t_0) = h(\vec x),& \vec x\in\RR^n\setminus B,
    \end{cases}
\end{align}
where \(\hat L\) is a fixed linear operator describing the boundary condition at the obstacle.

Since the initial data \(f,h\) has spatially compact support \(\{{\rm supp}(f),{\rm supp}(h)\subset \{|\vec x|^2_{\RR^n}<R\}\) for some \(R>0\), by causality argument, the support of the solution \(u\) is restricted to a causality cone: \({\rm supp}(u)\subset \{(\vec x,t)\in\RR^{n,1}\,|\, |\vec x|^2_{\RR^n}<(t+R-t_0)^2, t\geq t_0\}\). 
Without loss of generality, we may assume \(t_0 > R\) by a temporal translation.
In particular, the spacetime domain of interest is contained in a causality cone with apex at the origin. 
Therefore, we can restrict the domain of \eqref{eq:IBVPOriginal} to

\begin{align}
\label{eq:omega_unbounded}
    \Omega = \bigl((\RR^n\setminus B)\times [t_0,\infty)\bigr) \cap \left\{(\vec x, t)\,\middle\vert\,|\vec x|_{\RR^n}^2<t^2\right\},
\end{align}

\begin{figure}[H] 
\centering
\begin{picture}(240,120)
\put(60,0){\includegraphics[width=.8\linewidth]{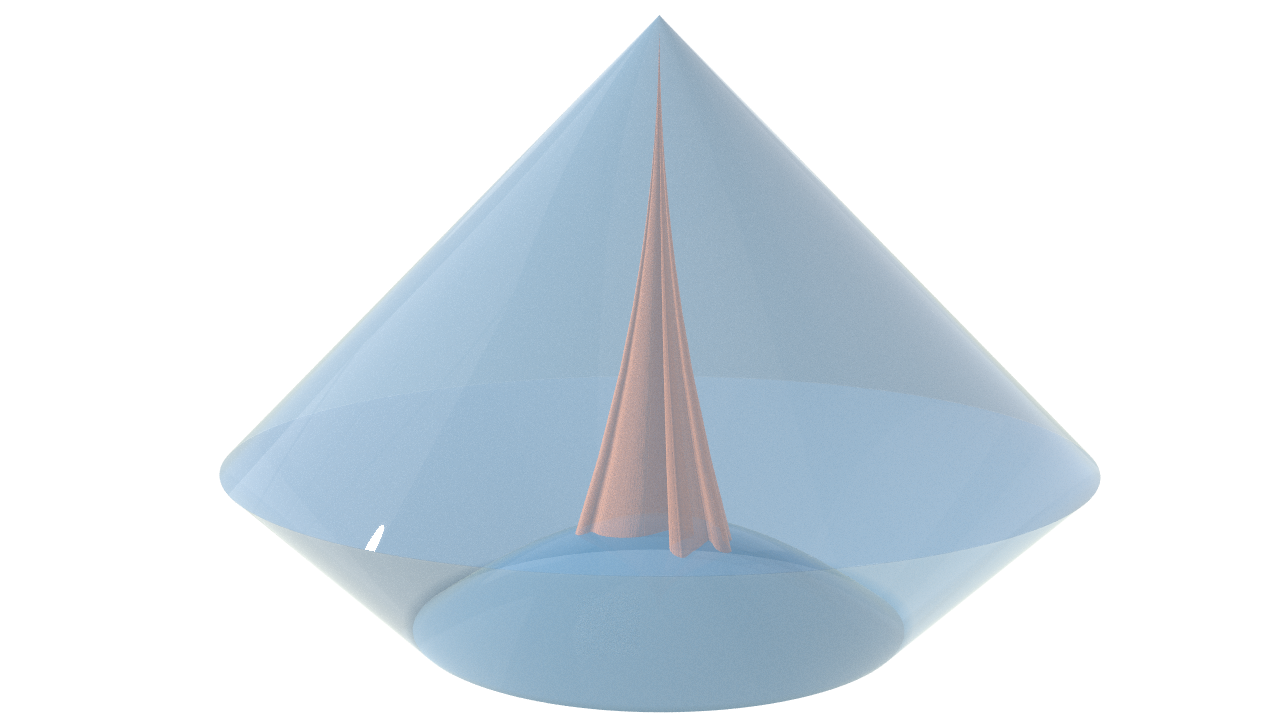}}
\put(-10,50){\includegraphics[width=.55\linewidth]{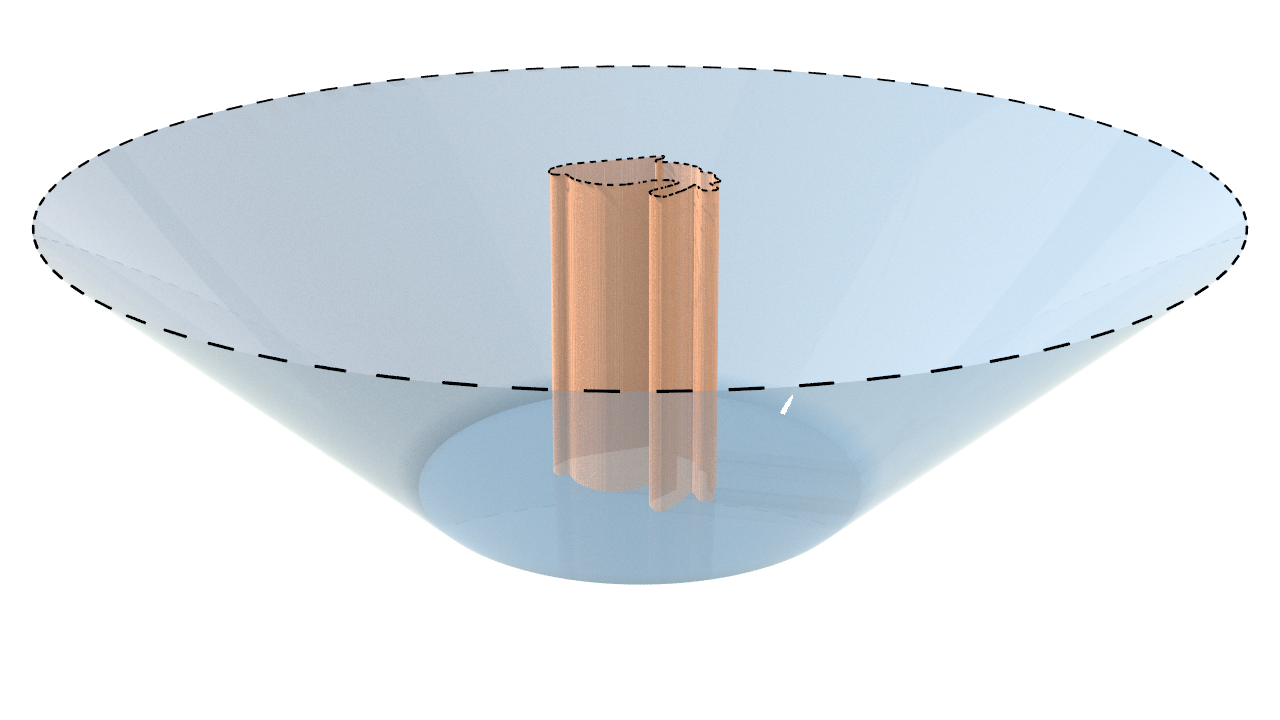}}
\put(0,80){\(\Omega\)}
\put(40,58){\(S_{\rm init}\)}
\put(69,100){\(S_{\rm obs}\)}
\put(91,20){\(\Omega'\)}
\put(118,0){\(S_{\rm init}'\)}
\put(174,45){\(S_{\rm obs}'\)}
\end{picture}
\caption{An unbounded Minkowski spacetime domain $\Omega$ with $n=2$ spatial dimensions contained in a causality cone with obstacle boundary $S_{\rm obs}$ and initial boundary $S_{\rm init}$~(left) is mapped to a bounded Minkowski spacetime domain $\Omega'$ with inverted obstacle boundary $S_{\rm obs}'$ and inverted initial boundary $S_{\rm init}'$~(right). }
\label{fig:boundedDomain}
\end{figure}
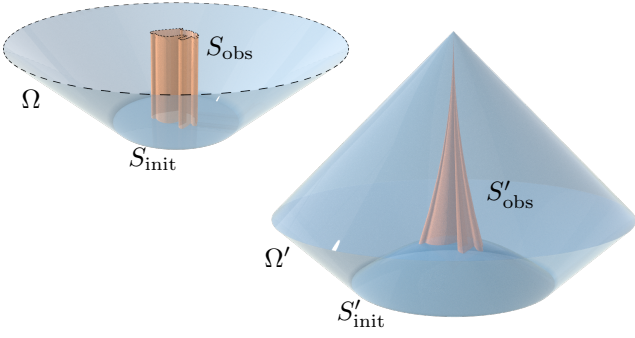

on which one poses the IBVP equivalent to \eqref{eq:IBVPOriginal}
\begin{align}
\label{eq:IBVPOnOmega}
    \begin{cases}
    \square_{(\vec x,t)}u(\vec x,t) = 0,&(\vec x,t)\in\Omega,\\
    (\hat Lu)(\vec x,t) = p(\vec x),& (\vec x,t)\in S_{\rm obs},\\
    u(\vec x,t_0) = f(\vec x),&(\vec x,t_0)\in S_{\rm init},\\
    \frac{\partial u}{\partial t}(\vec x, t_0) = h(\vec x),& (\vec x,t_0)\in S_{\rm init},
    \end{cases}
\end{align}
where 
\begin{align}
    S_{\rm obs} \coloneqq  \partial B\times(t_0,\infty)\subset\partial\Omega,\quad
    S_{\rm init} \coloneqq \partial\Omega\cap\{t=t_0\}.
\end{align}
Note that the remaining boundary component \(\partial\Omega\setminus(S_{\rm obs}\cup S_{\rm init}) = \partial \Omega\cap\{|\vec x|_{\RR^n}^2=t^2\}\) is light-like which requires no additional boundary condition for the wave~\(u\).

\subsection{Bounded form}

We apply \thmref{thm:MinkowskiKelvin} to transform \eqref{eq:IBVPOnOmega} to a bounded domain problem.
Let the M\"{o}bius transformation~$\Phi$ in \eqref{eq:mk_transform_invariance} be the Minkowski inversion map $\Inv$ given by \eqref{eq:Inv} with conformal factor $\varphi_{\Inv}$ given by \eqref{eq:Inv_factor}.
Then the inverted domain \(\Omega' \xleftrightarrow{\Inv}\Omega\) is a bounded domain in \(\RR^{n,1}_{(\vec\xi,\tau)}\). 
\figref{fig:boundedDomain} shows $\Omega$ for~$n = 2$ spatial dimensions, an obstacle boundary~$S_{\rm obs}$ of arbitrary shape, and an initial boundary~$S_{\rm init}$, as well as their inverted counterparts $\Omega'$, $S_{\rm obs}'$, and~$S_{\rm init}'$.
Using the following variable substitutions 
\begin{subequations}
\begin{align}
    &U(\vec\xi,\tau) \coloneqq  (u\circ \Inv)(\vec\xi,\tau),\quad
    G(\vec\xi,\tau)\coloneqq \varphi_\Inv^{-\frac{n-1}{2}}(\vec\xi,\tau), \label{eq:U}\\
    &V(\vec\xi,\tau) \coloneqq (\cK_\Inv u)(\vec\xi,\tau) = U(\vec\xi,\tau)/G(\vec\xi,\tau)\label{eq:V},
\end{align}
\end{subequations}
the transformed wave equation becomes
\begin{equation}\label{eq:mk_transform_inversion}
    \square_{(\vec \xi,\tau)} V(\vec\xi,\tau) = 0,\quad(\vec\xi,\tau)\in\Omega'.
\end{equation}
To complete the transformed IBVP, we describe the initial and boundary conditions.
The boundary component \(S_{\rm init}\) for the initial conditions is mapped to a hyperboloid surface 
\begin{align}
    S_{\rm init}'\coloneqq \Inv(S_{\rm init})=\left\{(\vec\xi,\tau)\in \partial\Omega\,\middle\vert\,
    \tfrac{\tau}{|\vec\xi|_{\RR^n}^2 - \tau^2} = t_0
    \right\}.
\end{align}
On \(S'_{\rm init}\) the initial data for \(V\) are given by
\begin{align}
    V(\vec\xi,\tau) = \tfrac{U(\vec\xi,\tau)}{G(\vec\xi,\tau)} = \left(\tfrac{f\circ\Inv}{G}\right)(\vec\xi,\tau),\quad 
    (\vec\xi,\tau)\in S'_{\rm init}
\end{align}
and 
\begin{align}
\nonumber
    \tfrac{\partial V}{\partial\tau}(\vec\xi,\tau) &=
    \left(
    \tfrac{\partial U}{\partial\tau}\tfrac{1}{G}
    -
    \tfrac{U}{G^2}\tfrac{\partial G}{\partial\tau}
    \right)(\vec\xi,\tau)\\
    &=
    \left(
    \tfrac{\partial U}{\partial\tau}\tfrac{1}{G}
    -
    \tfrac{f\circ\Inv}{G^2}\tfrac{\partial G}{\partial\tau}
    \right)(\vec\xi,\tau)
    \quad (\vec\xi,\tau)\in S'_{\rm init},
\end{align}
where the value \(\frac{\partial U}{\partial \tau}\vert_{S'_{\rm init}}\) is given in terms of \(f=u\vert_{S_{\rm init}}\) and \(h = \frac{\partial u}{\partial t}\vert_{S_{\rm init}}\) by
\begin{align}
\nonumber
\textstyle
    \frac{\partial U}{\partial \tau} &=
\textstyle
    \frac{\partial (u\circ\Inv)}{\partial \tau}
    =
    \frac{\partial u}{\partial t}\frac{\partial t}{\partial \tau} + 
    \sum_{i=1}^n \frac{\partial u}{\partial x_i}\frac{\partial x_i}{\partial \tau}\\
    &= 
\textstyle
    h\frac{\partial t}{\partial \tau} + 
    \sum_{i=1}^n \frac{\partial f}{\partial x_i}\frac{\partial x_i}{\partial \tau},\quad (\vec\xi,\tau)\in S'_{\rm init}.
\end{align}
and \(S_{\rm obs}\) for the obstacle is mapped to \(S_{\rm obs}' \coloneqq\Inv(S_{\rm obs})\).
The boundary condition for \(V\) at \(S_{\rm obs}'\) takes the form 
\begin{align}
    (\hat L'V)(\vec\xi,\tau) = q(\vec x,\tau),\quad(\vec x,\tau)\in S_{\rm obs}'
\end{align}
for some linear operator \(\hat L'\) and boundary data \(q\) derived similarly using chain rule.
The remaining boundary components~\(\partial\Omega'\setminus(S_{\rm init}'\cup S_{\rm obs}')\) are light-like as shown in \figref{fig:boundedDomain}, which do not require imposing boundary conditions.

In summary, the transformed IBVP takes the following bounded form: 

\begin{align}
\label{eq:ibvp_bounded}
    \begin{cases}
    \square_{(\vec \xi,\tau)} V(\vec \xi,\tau) = 0, & (\vec \xi, \tau) \in \Omega',
    \\
    (\hat{L}'V)(\vec \xi,\tau) = q(\vec \xi,\tau), & (\vec \xi, \tau) \in S_{\rm obs}', 
    \\
    V(\vec \xi,\tau) = \frac{f\circ \Inv}{G}(\vec \xi,\tau), & (\vec \xi, \tau) \in S_{\rm init}',
    \\
    \frac{\partial V}{\partial \tau}(\vec \xi,\tau) = \left(
    \tfrac{\partial U}{\partial\tau}\tfrac{1}{G}
    -
    \tfrac{f\circ \Inv}{G^2}\tfrac{\partial G}{\partial\tau}
    \right)(\vec \xi,\tau), &(\vec \xi, \tau) \in S_{\rm init}'.
    \end{cases}
\end{align}

\subsection{Numerical Scheme}

We describe the general numerical scheme for solving \eqref{eq:ibvp_bounded} and mapping the solution to the unbounded domain $\Omega$. To solve \eqref{eq:ibvp_bounded}, one first samples~$\Omega'$ at discrete points and then determines $V(\vec \xi, \tau)$ and $\frac{\partial V}{\partial \tau}(\vec \xi, \tau)$ at the points near $S_{\rm init}'$ for predefined values of $f(\vec x)$ and $h(\vec x)$. Then, $V(\vec \xi, \tau)$ is computed using a numerical integration method such as the FDTD method. See \appref{app:cost_mk_transform} for a description of the computational cost of solving \eqref{eq:ibvp_bounded} using the FDTD method in $(n+1)$-dimensional spacetime.

We query the solution \(u(\vec x,t)\) at any point \((\vec x,t)\) on the unbounded domain by finding \((\vec\xi,\tau) = \Inv(\vec x,t)\) and evaluating \(u(\vec x,t) = G(\vec\xi,\tau)V(\vec\xi,\tau)\) where \(G(\vec\xi,\tau)\) is given analytically and \(V(\vec\xi,\tau)\) is interpolated from the solution grid. The general solver for the infinite domain problem is summarized by Algorithm \ref{alg:KelvinWave2}.

\begin{algorithm}[H]
    \caption{\small Solver for infinite domain problem}
    \label{alg:KelvinWave2}
    \small
    \begin{algorithmic}[1]
        \Require Infinite domain problem \eqref{eq:IBVPOnOmega}. Set of query points \(Q=\{(\vec x_i,t_i)\}\subset \Omega\).
        \State \(V(\vec \xi, \tau)\gets\) Solve the bounded domain problem \eqref{eq:ibvp_bounded} with any numerical method.
        \ForEach {\((\vec x_i,t_i)\in Q\)}
        \State \((\vec\xi_i,\tau_i)\gets\Inv(\vec x_i,t_i)\)
        \State \(u(\vec x_i,t_i) \gets G(\vec\xi_i,\tau_i)\cdot \textsc{Evaluate}(V,(\vec\xi_i,\tau_i)) \)
        \EndFor
        \Ensure \(u(\vec x_i, t_i)\) for each query \((\vec x_i,t_i)\in Q\).
    \end{algorithmic}
\end{algorithm}

The accuracy of our numerical solution to $V(\vec \xi, \tau)$ depends on the resolution of the grid that we constructed on~$\Omega'$. For a rectilinear spacetime grid, the error in the solution corresponding to the interior region of $\Omega'$ (excluding the obstacle boundary) is given by~\cite{langtangen2017finite} 

\begin{equation}
    \textnormal{error} = \mathcal{O}(\Delta \xi^2) + \mathcal{O}(\Delta \tau^2).
\end{equation}
where $\Delta \xi$ is the distance between adjacent points along each spatial dimension of the grid, and $\Delta \tau$ is the separation between adjacent temporal grid points. While the grid  on $\Omega'$ is distorted by the M\"{o}bius inversion that maps $\Omega'$ to $\Omega$, the solution to the wave equation is not distorted by the inversion (\figref{fig:kelvin_wave_accuracy}). In other words, if $V(\vec \xi, \tau)$ is a smooth function, then $u(\vec x, t)$ is also a smooth function.
Moreover, a particular wavelength on $\Omega'$ remains the same wavelength on $\Omega$ after the inversion. Therefore, no numerical errors are introduced to our solution by the M\"{o}bius inversion that maps $V(\vec \xi, \tau)$ to $u(\vec x, t)$.

   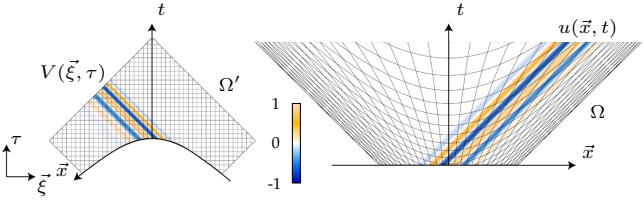
\begin{figure}[H]
\centering
\begin{picture}(240,90)
\put(0,0){\includegraphics[width=0.97\linewidth]{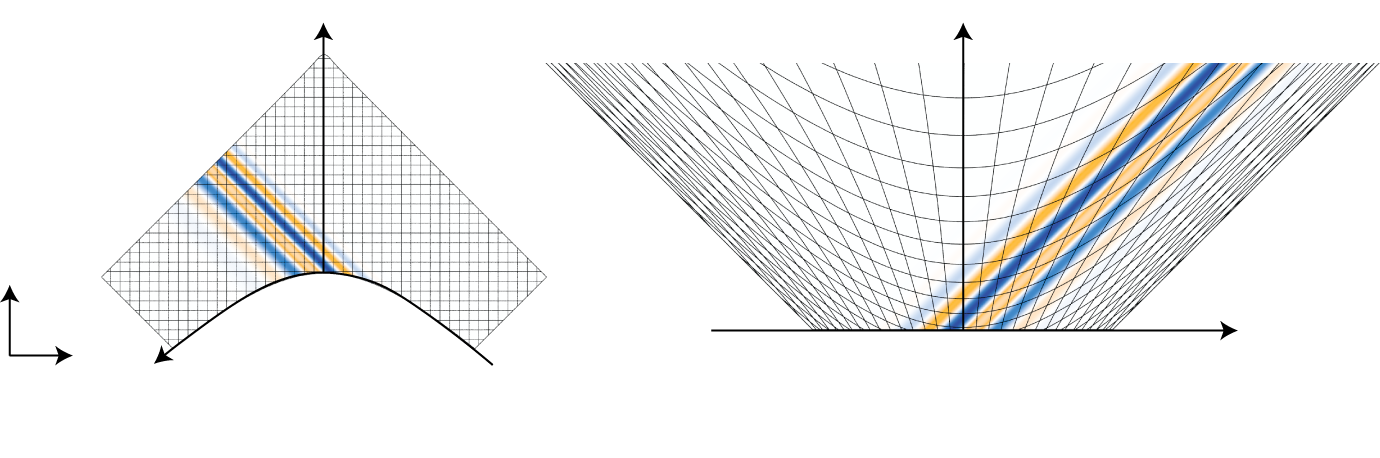}}
\put(220,40){\scriptsize \(\Omega\)}
\put(81,50){\scriptsize \(\Omega'\)}
\put(217,25){\scriptsize \(\vec x\)}
\put(20,18){\scriptsize \(\vec x\)}
\put(170,79){\scriptsize \(t\)}
\put(58,79){\scriptsize \(t\)}
\put(3,30){\scriptsize \(\tau\)}
\put(13,12){\scriptsize \(\vec\xi\)}
\put(208,72){\scriptsize \(u(\vec x,t)\)}
\put(14,55){\scriptsize \(V(\vec\xi,\tau)\)}
\put(96,10){\begin{tikzpicture}
\pgfplotscolorbardrawstandalone[ 
    colormap={myProteinColor}{
        rgb255=(0, 0, 187)
        rgb255=(0, 136, 255)
        rgb255=(255, 255, 255)
        rgb255=(255, 187, 0)
        rgb255=(255, 225, 187)
    },
    point meta min=-1,
    point meta max=1,
    colorbar style={
        height=30pt,
        width=3pt,
        ytick={-1,-0.5,0,0.5,1},
        ytick style = {draw=none},
        yticklabels={{-1},{},{ 0},{},{ 1}},
        yticklabel style={font=\tiny, style={xshift=-3.5ex, anchor=west}}
    }]
\end{tikzpicture}}
\end{picture}
\caption{A smooth function on the bounded Minkowski spacetime domain $\Omega'$ (left) remains smooth after the M\"{o}bius inversion maps the function to the unbounded Minkowski spacetime domain~$\Omega$~(right).}
\label{fig:kelvin_wave_accuracy}
\end{figure}

\subsection{Solution} \label{sec:Solution}

We present our numerical solution to \eqref{eq:IBVPOnOmega} on an unbounded domain~$\Omega$ with~$n = 2$ spatial dimensions and a particular choice of boundary condition on~$S_{\rm obs}$ and initial conditions on~$S_{\rm init}$ (\figref{fig:boundedDomain}). A closed obstacle boundary~$S_{\rm obs}$ of arbitrary shape and perfectly reflective material (i.e. $\hat{L} = 1$, $p = 0$) is inserted near the origin of $\Omega$. On the bounded domain~$\Omega' = \Inv(\Omega)$, the boundary condition becomes

\begin{equation}\label{eq:BC_bounded_domain}
     V(\vec\xi,\tau) = 0,\quad(\vec\xi,\tau)\in S'_{\rm obs}.
\end{equation}

To radiate energy symmetrically in all directions from an initial pulse, we set the initial data on~$S_{\rm init}$ as 

\begin{subequations}  \label{eq:initFns}
\begin{gather} 
f(\vec x) = A\,\, \textnormal{exp}\sum_{i=1}^{2}-\frac{(x_i-\mu_i)^2}{2\sigma^2},\label{eq:GaussianFn}
\\
h(\vec x) = 0 \label{eq:symmetry},
\end{gather}
\end{subequations}
for real values of $A$, $\sigma$, and $\mu_i$. To solve \eqref{eq:ibvp_bounded}, we construct a rectilinear spacetime grid on~$\Omega'$ composed of $M=405^2$ spatial grid points and $N=406$ time steps that is symmetric about the $\tau$-axis and then iteratively compute $V(\vec \xi, \tau)$ over the grid using the FDTD method. Then, we map $V(\vec \xi, \tau)$ to $u(\vec x, t)$ using Algorithm~\ref{alg:KelvinWave2}. \figref{fig:InfSoln} plots four successive time steps of our solution to $u(\vec x, t)$.

\section{Conclusion}
In this paper, we presented a mathematical approach for numerically solving wave propagation problems on unbounded spacetime domains using only finite computation time and memory. Our method assumes that waves propagate in a homogeneous, motionless medium with no viscothermal dissipation, as described by the scalar wave equation~\eqref{eq:wavePDE}. It is an open question whether there exist transformations of other types of wave equations which preserve the equation of motion but not the domain boundedness, such as wave equations that model viscothermal losses and propagation in inhomogeneous and turbulent mediums, vector wave equations, or gravitational wave equations. Regarding the implementation of our method presented in Section~\ref{sec:InfDomainProb}, we assumed that waves propagate in two spatial dimensions and that the obstacle boundary is perfectly reflective. Extending our implementation to three spatial dimensions and incorporating different obstacle boundary impedances could be useful for simulating more realistic physical scenarios of scalar wave propagation.

\section{Acknowledgements}

This work was partially funded by NSF CAREER Award 2239062, the UC San Diego Department of Music, and the UC San Diego Center for Visual Computing. Additional support was provided by SideFX software.

\appendix

\section{Conformality of the Minkowski Inversion Map}\label{app:inversionConformalityProof}

\begin{proof}
 We show that the Minkowski inversion map $\Phi = \Inv$ given by~\eqref{eq:Inv} is conformal with conformal factor $\varphi_{\Inv}$ given by (\ref{eq:Inv_factor}). Consider two arbitrary vectors $\vec u, \vec v$ in a Minkowski spacetime $\mathbb R^{n,1}$ with metric $\eta$ given by~\eqref{eq:Minkowski_metric}. Note that the Jacobian of $\Phi$ is given by

\begin{equation} 
d\Phi = \varphi_\Inv(\mathbb I_{n+1} - 2\varphi_\Inv \vec v \vec v^T \mathbb I_{n,1}).
\end{equation}

\noindent
 Then, we deduce
\begin{align} 
   (\Phi^* \eta) (\vec u, \vec v) & = \eta(d\Phi(\vec u), d\Phi(\vec v)) \nonumber
   \\
   & = \vec u^\intercal d\Phi^\intercal \mathbb I_{n,1} d\Phi \vec v \nonumber 
   \\
   & = \varphi_{\Inv}^2 \eta(\vec u, \vec v) \nonumber,
\end{align}

\noindent
where $\Phi^*$ is the pullback by $\Phi$, $\mathbb I_{n+1} = \operatorname{diag}(1,\ldots, 1)$ is the identity matrix, and $\mathbb I_{n,1} = \operatorname{diag}(1,\ldots,1,-1)$ is the Minkowski metric matrix. 
\end{proof}

\section{Extended Minkowski Spacetime}\label{app:ExtendedMinkowskiSpacetime}
The inversion as is displayed here is defined only for points away from the \emph{light cone} \(C_0=\{(\vec x,t)\in\RR^{n,1}\,\vert\,|\vec x|_{\RR^n}^2-t^2=0\}\).
To define inversion for points on the light cone,
consider the \emph{one-cone extension} of the Minkowski spacetime
\begin{align}
\label{eq:ExtendedMinkowskiSpacetime}
    \overline{\RR^{n,1}}\coloneqq \RR^{n,1}\sqcup C_\infty,
\end{align}
where \(C_\infty\) is a copy of \(C_0\) with an identification map \(j\colon C_0\xrightarrow{\simeq} C_\infty\).  Define the restriction of inversion on \(C_0\subset\RR^{n,1}\) and \(C_\infty\) as swapping \(C_0\) and \(C_\infty\) via the map \(j\).  Define the topology of \(\overline{\RR^{n,1}}\) such that the inversion map is continuous.
Extend Lorentz transformations, translations and scaling to \(\overline{\RR^{n,1}}\) continuously.

Every conformal map \(\Phi\colon\Omega\to\Omega'\) extends to a global conformal map \(\Phi\colon\overline{\RR^{n,1}}\to\overline{\RR^{n,1}}\) whose conformal factor is allowed to vanish or be infinite \(\varphi\colon\overline{\RR^{n,1}}\to [0,\infty]\).

\section{Invariance of the Wave Equation under Minkowski-Kelvin Transforms} \label{app:conformWaveEqnMinkowski}

\begin{proof}
Here we show \thmref{thm:MinkowskiKelvin}.
Since the wave equation is invariant under Minkowski-isometries (Lorentz transformations), translations, and scaling in the domain, we only need to show~\eqref{eq:mk_transform_invariance} for the inversion map $\Phi(\vec \xi, \tau) = \Inv(\vec \xi,\tau)$~\eqref{eq:Inv} with conformal factor $\varphi_{\Inv}$~\eqref{eq:Inv_factor}. We start with the assumption that~\eqref{eq:wavePDE} is valid. In exterior calculus notation, \eqref{eq:wavePDE} and \eqref{eq:mk_transform_invariance} are respectively written as

\begin{subequations}
\begin{align}
d\star_1 du & = 0, \label{eq:wave_eqn_xt}
\\
d\star_1' d (\mathcal K_\Inv u) & = 0.\label{eq:wave_eqn_xitau}
\end{align}
\end{subequations}
Here, \(\star_1\) is the Hodge star of the Minkowski spacetime \(\RR^{n,1}_{(\vec x, t)}\), and \(\star_1'\) is the Hodge star of the corresponding inverted Minkowski spacetime \(\RR^{n,1}_{(\vec\xi,\tau)}\). Applying the pullback map $\Inv^\ast$ to~\eqref{eq:wave_eqn_xt}, we get

\begin{equation}
d\Tilde{\star}_1' dU = 0, \label{eq:dtilde*dU}
\end{equation}
where \(\tilde\star_1' = \Inv^*\star_1\) is a Hodge star that respects the metric of \(\RR^{n,1}_{(\vec x,t)}\) and $U = \Inv^*u = u\circ\Inv$. The Hodge stars \(\Tilde{\star}_1'\) and \(\star_1'\) are related by
\begin{equation}
\label{eq:star_tilde}
\Tilde{\star}_1' = \frac{1}{G^2}\star_1',
\end{equation}
where $G = \varphi_\Inv^{-\frac{n-1}2}$. Inserting~\eqref{eq:star_tilde} into~\eqref{eq:dtilde*dU} and then differentiating, we obtain 

\begin{equation}
\label{eq:d*dU}
d\star_1'd U = -2G dU \wedge \star_1' d\frac{1}{G}.
\end{equation}
We now expand $d\star_1' d (\mathcal K_\Inv u)$ from~\eqref{eq:wave_eqn_xitau} and use~\eqref{eq:d*dU} to simplify the resulting expression as follows
\begin{subequations}
\begin{align}
d\star_1' d (\mathcal K_\Inv u) & = d\star_1' d \frac{U}{G},
\\
& = \big(d\star_1' d \frac{1}{G}\big)U + 2dU \wedge \star_1' d\frac{1}{G} + \frac{1}{G}d\star_1'd U,
\label{eq:d*dKu_second_line}\\
& = \big(d\star_1' d \frac{1}{G}\big)U.
\label{eq:d*dKu_third_line}
\end{align}
\end{subequations}
Finally, expressing the d'Alembertian in $n$-dimensional spherical coordinates with $r = |\vec \xi|_{\RR^n}$, one can verify the following:
\begin{subequations}
\begin{align}
d\star_1' d \frac{1}{G} & = \bigg(\frac{\partial^2}{\partial r^2} + \frac{n-1}{r} \frac{\partial}{\partial r} - \frac{\partial^2}{\partial \tau^2}\bigg) \frac{1}{G} ,
\\
& = 0.
\end{align}
\end{subequations}
Thus, \eqref{eq:mk_transform_invariance} is valid for the inversion map.

\end{proof}

\section{Computational Cost of Bounded Domain Solver}
\label{app:cost_mk_transform}

We describe the computational cost of solving~\eqref{eq:ibvp_bounded} using the FDTD method. Let $\Omega$ be an $(n+1)$-dimensional unbounded Minkowski spacetime domain given by~\eqref{eq:omega_unbounded} such that $\Omega' = \Inv(\Omega)$ is a bounded Minkowski spacetime domain. The number of spatial grid points $M$ and time steps $N$ used to sample the entire bounded domain $\Omega'$ on a rectilinear spacetime grid with spatial sampling interval $\Delta \xi$ and time step $\Delta \tau$ are given by

\begin{subequations}  
\begin{align} 
M & \approx \left\lceil \frac{\xi_0 + \tau_0}{\Delta \xi} \right\rceil^n \label{eq:M},
\\
N & \approx \left\lceil \frac{\tau_0}{\Delta \tau} \right\rceil \label{eq:N},
\end{align}
\end{subequations}
 where $\xi_0$ and $\tau_0$ are defined as
\begin{subequations}  \label{eq:xi_tau_0}
\begin{align} 
\xi_0 &\coloneqq \bigg|\frac{x_0}{x_0^2 - t_0^2}\bigg|,
\\
\tau_0 &\coloneqq \bigg|\frac{t_0}{x_0^2 - t_0^2}\bigg|,
\end{align}
\end{subequations}
and $x_0$ is the radius of the event horizon at the initial time~$t_0$ on~$\Omega$. Note that $x_0$ and $t_0$ are chosen such that $t_0 > x_0$. The units of storage as well as the floating-point operations required to solve~\eqref{eq:ibvp_bounded} both scale as $\mathcal{O}(MN)$.

\end{multicols}

\begin{figure*}[!t]
\centering
\begin{picture}(510,100)
\put(-20,0){\includegraphics[width=0.23\textwidth, trim=190px 250px 450px 0, clip]{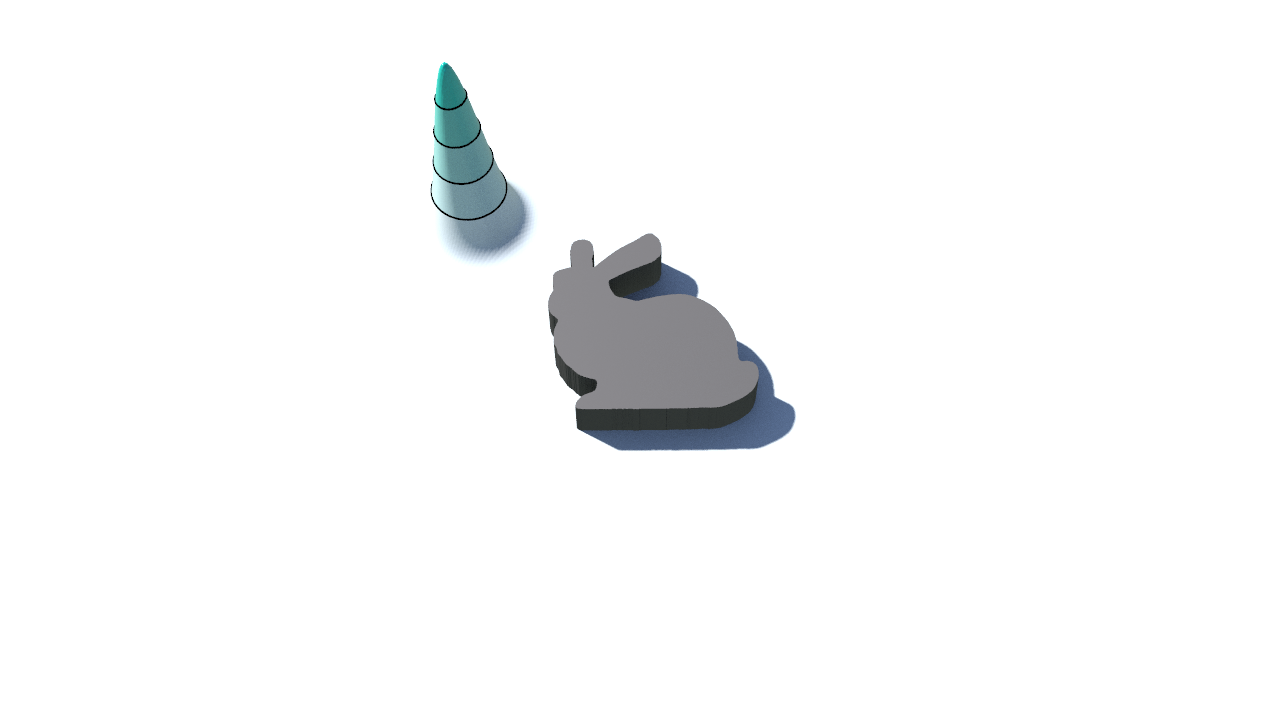}}
\put(90,0){\includegraphics[width=0.23\textwidth, trim=190px 250px 450px 0, clip]{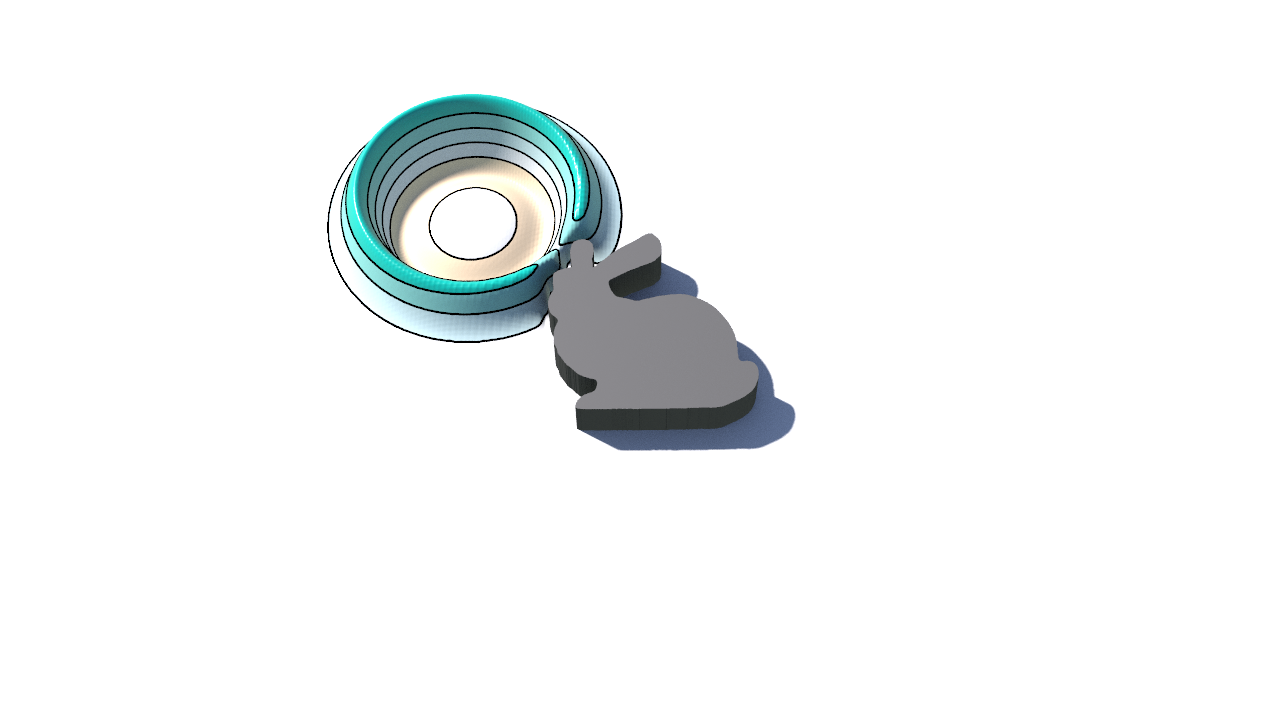}}
\put(210,0){\includegraphics[width=0.23\textwidth, trim=190px 250px 450px 0, clip]{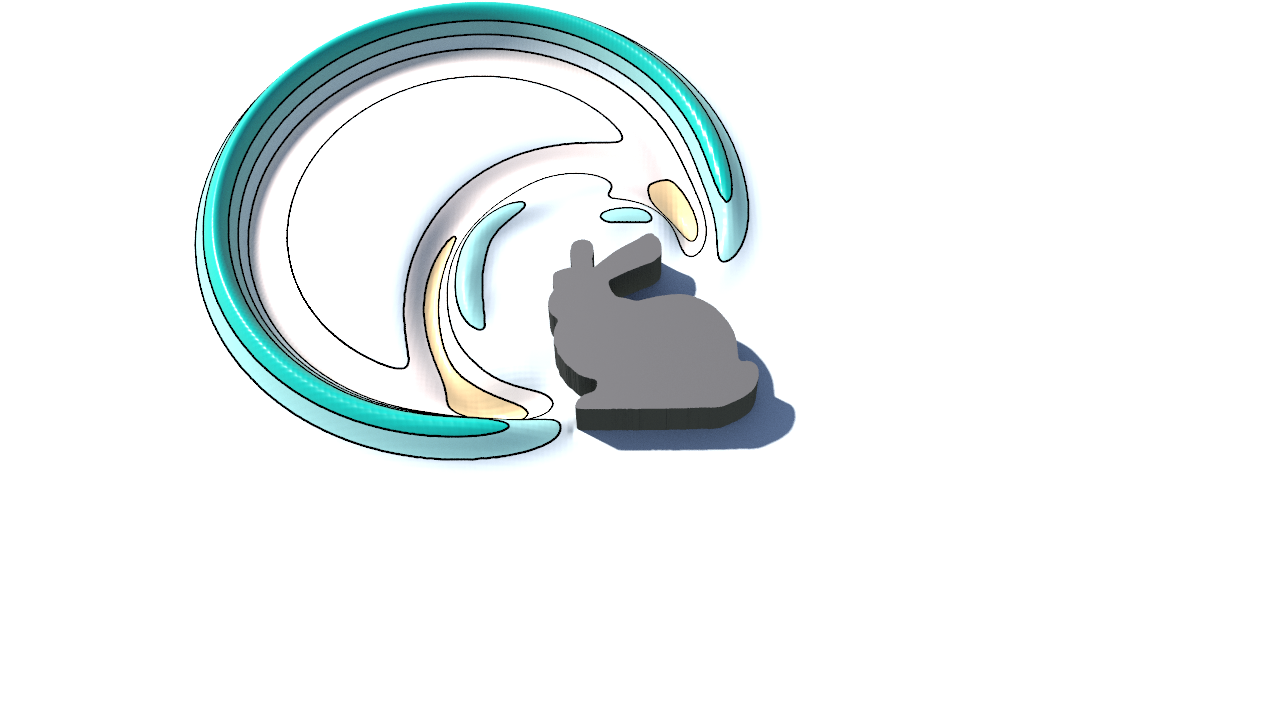}}
\put(320,0){\includegraphics[width=0.36\textwidth]{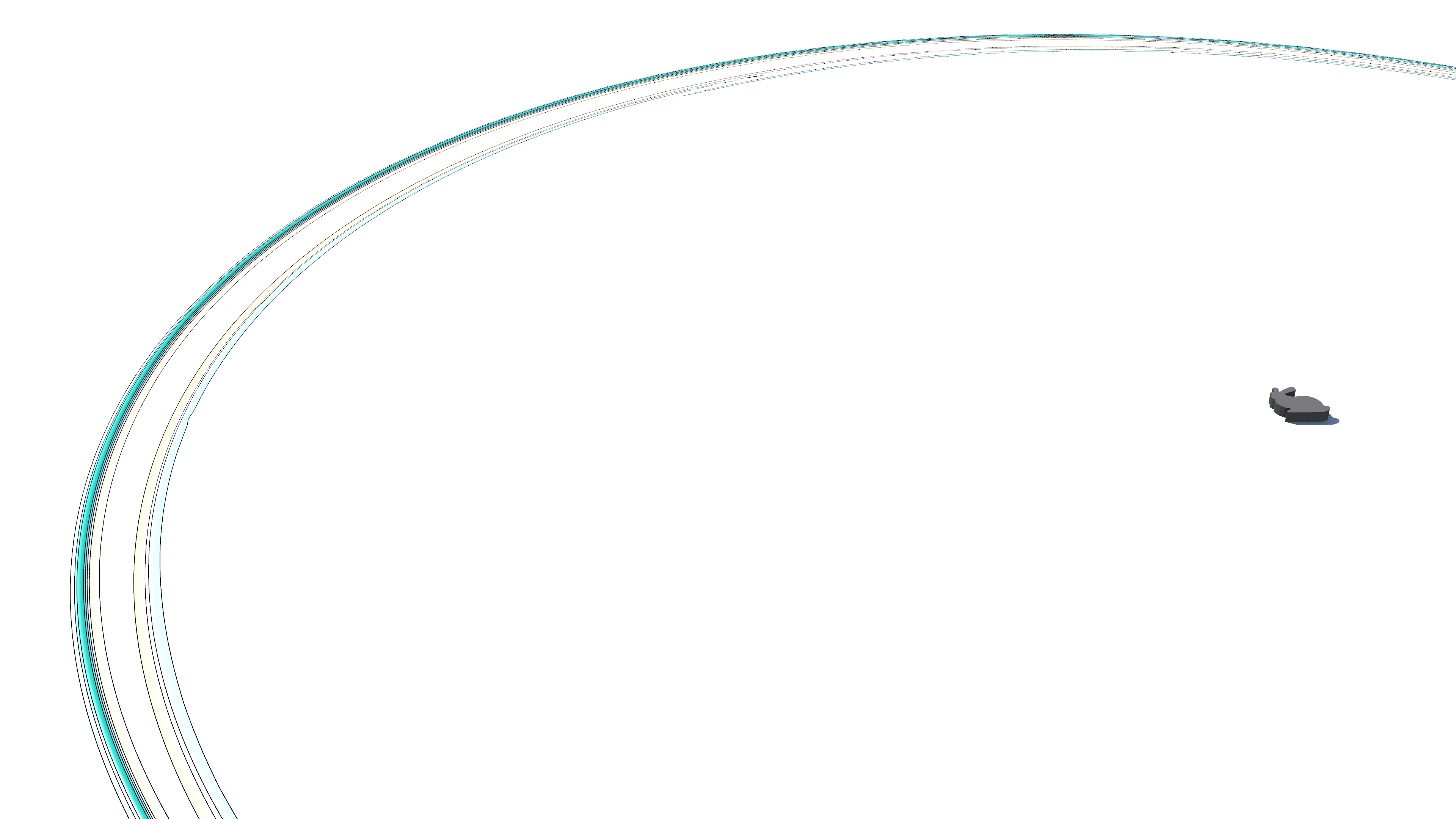}}
\end{picture}

\caption{Numerical solution to the infinite domain problem \eqref{eq:IBVPOnOmega} on an unbounded Minkowski spacetime domain $\Omega$ with $n=2$ spatial dimensions and a perfectly reflective obstacle boundary placed near the $t$-axis. The simulation is excited using a symmetric Gaussian pulse. Four successive time steps of the solution are plotted from left to right.} 

\label{fig:InfSoln}
\end{figure*}



\pagebreak

\bibliographystyle{amsplain}
\bibliography{KelvinWaveBib}

\providecommand{\bysame}{\leavevmode\hbox to3em{\hrulefill}\thinspace}
\providecommand{\MR}{\relax\ifhmode\unskip\space\fi MR }
\providecommand{\MRhref}[2]{%
  \href{http://www.ams.org/mathscinet-getitem?mr=#1}{#2}
}
\providecommand{\href}[2]{#2}
\begin{thebibliography}{10}

\bibitem{Bateman:1910:TEE}
Harry Bateman, \emph{The transformation of the electrodynamical equations},
  Proceedings of the London Mathematical Society \textbf{2} (1910), no.~1,
  223--264.

\bibitem{berenger1994perfectly}
Jean-Pierre Berenger, \emph{A perfectly matched layer for the absorption of
  electromagnetic waves}, Journal of computational physics \textbf{114} (1994),
  no.~2, 185--200.

\bibitem{boore1972finite}
David~M Boore, \emph{Finite difference methods for seismic wave propagation in
  heterogeneous materials}, Methods in computational physics \textbf{11}
  (1972), 1--37.

\bibitem{cahen1983domaines}
M~Cahen and Y~Kerbrat, \emph{Domaines sym{\'e}triques des quadriques
  projectives}, J. Math. Pures Appl \textbf{9} (1983), no.~62, 327--348.

\bibitem{cecil2008lie}
Thomas~E Cecil, \emph{Lie sphere geometry}, Springer, 2008.

\bibitem{chern2019reflectionless}
Albert Chern, \emph{A reflectionless discrete perfectly matched layer}, Journal
  of Computational Physics \textbf{381} (2019), 91--109.

\bibitem{Cunningham:1910:PRE}
Ebenezer Cunningham, \emph{The principle of relativity in electrodynamics and
  an extension thereof}, Proceedings of the London Mathematical Society
  \textbf{2} (1910), no.~1, 77--98.

\bibitem{Dirac:1936:WEC}
Paul Adrien~Maurice Dirac, \emph{Wave equations in conformal space}, Annals of
  Mathematics \textbf{37} (1936), no.~2, 429--442.

\bibitem{engquist1977absorbing}
Bj{\"o}rn Engquist and Andrew Majda, \emph{Absorbing boundary conditions for
  numerical simulation of waves}, Proceedings of the National Academy of
  Sciences \textbf{74} (1977), no.~5, 1765--1766.

\bibitem{fix1978variational}
George~J Fix and Samuel~P Marin, \emph{Variational methods for underwater
  acoustic problems}, Journal of Computational Physics \textbf{28} (1978),
  no.~2, 253--270.

\bibitem{huang1991scalar}
WP~Huang, ST~Chu, A~Goss, and SK~Chaudhuri, \emph{A scalar finite-difference
  time-domain approach to guided-wave optics}, IEEE Photonics Technology
  Letters \textbf{3} (1991), no.~6, 524--526.

\bibitem{Klein:1926:VUH}
Felix Klein and Wilhelm Blaschke, \emph{Vorlesungen {\"u}ber h{\"o}here
  geometrie}, vol.~22, J. Springer, 1926.

\bibitem{langtangen2017finite}
Hans~Petter Langtangen and Svein Linge, \emph{Finite difference computing with
  pdes: a modern software approach}, Springer Nature, 2017.

\bibitem{Lomont:1961:CIM}
JS~Lomont, \emph{Conformal invariance of massless dirac-like wave equations},
  Il Nuovo Cimento (1955-1965) \textbf{22} (1961), 673--679.

\bibitem{Mclennan:1956:CIC}
JA~McLennan, \emph{Conformal invariance and conservation laws for relativistic
  wave equations for zero rest mass}, Il Nuovo Cimento (1955-1965) \textbf{3}
  (1956), 1360--1379.

\bibitem{mehra2014acoustic}
Ravish Mehra, Nikunj Raghuvanshi, Anish Chandak, Donald~G Albert,
  D~Keith~Wilson, and Dinesh Manocha, \emph{Acoustic pulse propagation in an
  urban environment using a three-dimensional numerical simulation}, The
  Journal of the Acoustical Society of America \textbf{135} (2014), no.~6,
  3231--3242.

\bibitem{nabizadeh2021kelvin}
Mohammad~Sina Nabizadeh, Ravi Ramamoorthi, and Albert Chern, \emph{Kelvin
  transformations for simulations on infinite domains}, ACM Transactions on
  Graphics (TOG) \textbf{40} (2021), no.~4, 1--15.

\bibitem{orlanski1976simple}
I\_ Orlanski, \emph{A simple boundary condition for unbounded hyperbolic
  flows}, Journal of computational physics \textbf{21} (1976), no.~3, 251--269.

\bibitem{seibel2022boundary}
Daniel Seibel, \emph{Boundary element methods for the wave equation based on
  hierarchical matrices and adaptive cross approximation}, Numerische
  Mathematik \textbf{150} (2022), no.~2, 629--670.

\bibitem{thomson1845extrait}
William Thomson, \emph{Extrait d'une lettre de m. william thomson {\`a} m.
  liouville.}, Journal de math{\'e}matiques pures et appliqu{\'e}es \textbf{10}
  (1845), 364--367.

\bibitem{tsynkov1998numerical}
Semyon~V Tsynkov, \emph{Numerical solution of problems on unbounded domains. a
  review}, Applied Numerical Mathematics \textbf{27} (1998), no.~4, 465--532.

\bibitem{yee1966numerical}
Kane Yee, \emph{Numerical solution of initial boundary value problems involving
  maxwell's equations in isotropic media}, IEEE Transactions on antennas and
  propagation \textbf{14} (1966), no.~3, 302--307.

\end{thebibliography}

\end{document}